\documentclass[12pt, a4paper]{amsart}
\usepackage{amssymb,amscd, hyperref, epsf}
\usepackage{graphicx}

\let\kappa=\varkappa
\let\epsilon=\varepsilon
\let\phi=\varphi

\def\R{\mathbb R}

\newtheorem{thm}{Theorem}
\newtheorem{construction}[thm]{Construction}

\theoremstyle{definition}

\newtheorem{remark}[thm]{Remark}

\numberwithin{equation}{subsection}

\begin{document}

\title[Inequalities of Constructive Mathematics]{Inequalities and equalities of means in Constructive Mathematics}
\author[T. Kawakubo]{Teruma Kawakubo\\
\normalfont\small St. Mary's International School, 1-6-19 Seta, Setagaya-ku, Tokyo 158-0095, Japan\\
\normalfont\small\href{mailto:teruma.kawakubo@gmail.com}{teruma.kawakubo@gmail.com}\\
\normalfont\small August 2026}
\date{}

\subjclass{Primary 26E60, 26E40; Secondary 03F60} 

\begin{abstract}
    We study the inequalities and equalities of means of numbers in constructive mathematics. The parameters of such means are constructive real numbers. We show that it is impossible to have an algorithm that can always decide whether the two means are equal or not. Of course, this problem can always be solved if the parameters are rational numbers.
\end{abstract}

\maketitle

\section{Introduction}
The {\em Arithmetic Mean (AM)} is one of the most known means and is used daily. Additionally, other means such as the {\em Geometric Mean (GM)}, the {\em Harmonic Mean (HM)}, and the {\em Quadratic Mean (QM)} are also heavily used in data analysis and statistics.
To recall, let \(x_1\), \(x_2\), \dots, \(x_n\in \R^+\), then: \[AM = \frac{x_1 + x_2 + \cdots + x_n}{n}\] \[GM = \sqrt[n]{x_1 x_2 \cdots x_n}\] \[HM = \frac{1}{\frac{\frac{1}{x_1}+\frac{1}{x_2}+\cdots+\frac{1}{x_n}}{n}}=\frac{n}{\frac{1}{x_1}+\frac{1}{x_2}+\cdots+\frac{1}{x_n}}\] \[QM = \sqrt{\frac{x_1^2+x_2^2+\cdots+x_n^2}{n}}\]
Additionally, these means satisfy the inequality \(HM\le GM\le AM\le QM\).
More generalized means are also used in data analysis and statistics, and they satisfy the inequalities discussed later.

Now, the {\em generalized p-power mean} is a function defined as (for a real number \(p\)) \[M_p(x_1, x_2, \dots, x_n)= \left(\frac{x_1^p+x_2^p+\cdots+x_n^p}{n}\right)^\frac1p \] and in particular, when \(p=1\) it becomes the AM, for \(p=2\) it becomes the QM, for \(p=-1\) it becomes the HM, and for \(p=0\) it becomes the GM, see \cite{Hardy, Korovkin}.

Additionally, it is known that (see \cite{Hardy}) for \(p<q\), \[M_p(x_1, x_2, \dots, x_n)\le M_q(x_1, x_2, \dots, x_n)\] and equality is when \(x_1=x_2=\cdots=x_n\). This is also applied to a special case: the inequality of \(HM\le GM\le AM\le QM\), and the equality only occurs when \(x_1=x_2=\cdots=x_n\).

In {\em Constructive Mathematics}, the Principle of Excluded Middle, which states that \(P \lor \lnot P\) always holds, is not necessarily true. Instead, a constructive mathematician must specify the true part of the disjunction. Furthermore, proof by contradiction is not valid, as proving that an object can not not exist does not imply that the object exists, in particular because one does not know how to construct it. In the Russian school of constructive mathematics, the Markov Principle, developed by A. A. Markov \cite{Markov2} says that if it is known that a decidable subset S of \(\mathbb{N}\) can not be empty then you can find an element of it. (One checks the natural numbers one by one until you find the element of S. The weak point in this is that it can take thousands of years to find it and the American School of Constructive Mathematics does not allow this conclusion, see \cite{Bishop}) This is also called the {\em principle of constructive choice} and sometimes allows one to argue by contradiction.

Constructive Mathematics was also developed by E. Bishop \cite{Bishop}, in a different but largely similar way. However, Bishop's approach does not allow the use of Markov's principle, as briefly mentioned above. Thus, most theorems and constructions in Bishop's approach can be brought over to Markov's approach without much change, but the converse is not necessarily true.

A {\em Constructive Real Number (CRN)} is a pair of algorithms \((A, B)\). An algorithm \(A\) ({\em called the fundamental sequence}) transforms natural numbers into rational numbers that are members of a Cauchy sequence. Algorithm \(B\) ({\em convergence regulator}) transforms natural numbers into natural numbers and guarantees the speed of convergence in itself of the sequence \(A\). That is, for every integer \(n>0\) and \(i, j \ge B(n)\) then \(\left| A(i)-A(j)\right| < 2^{-n}\). If a \(CRN\) \((A, B)\) is denoted by \(x\), then \(x(n)\) denotes \(A(n)\).

The \(CRNs\) appear naturally, for example when you measure objects with better and better precision but you never know the exact measurement.

Classically, inequalities and equalities of the power means (and their special cases) are decidable. But, when the problem parameters are \(CRNs\) it is not always so. In Constructive Mathematics the situation is delicate, as one cannot always tell if a pair of \(CRNs\) are equal and this can be always done in classical mathematics and for \(CRNs\) that are rational numbers.

\begin{remark}\label{remark1}
    When the parameters are rational numbers, thought of as a pair of integer numbers (numerator and denominator), the equality and inequality of the means is decidable by simplifying the expressions and taking them to the common denominator.
\end{remark}

The proofs in this paper will show that the inequalities and equalities of the \(CRNs\) are not always decidable.

\section{Main Results}

Let \(P\) be a (partially defined) algorithm that transforms natural
numbers into 0 and 1 and that is not extendable to a totally defined
algorithm, see ~\cite{Kleene} and Theorem 12 of \cite{Vereshchagin}.
Note that the Halting Problem for this algorithm is undecidable and the set of numbers on which \(P\) gives output 1, as well as the set of numbers which it outputs 0, is undecidable as well. Theorem 13 of \cite{Vereshchagin} states that these two disjoint sets cannot be separated by a decidable set. (A separation by a decidable set is defined to be when a decidable set \(C\) contains one of the two disjoint sets \(X\) or \(Y\) and is disjoint to the other.)  
\begin{construction}\label{construction2} For each natural $n$ we define the following computer-generated sequence of rational numbers $a_n^1$.
\begin{itemize}
\item For each $k$ we put $a_n^1(k)=3+2^{-k}$ if the algorithm $P$ did not yet finish working on input $n$ by step $k$.
\item we put $a_n^1(k)=3+2^{-m}$ if $P$ has finished working on $n$ by step $k$. Here, $m$ is the step number when $P$ finished working.
\end{itemize}
\end{construction}

{\em For each $n$, the sequence $a_n^1$ converges in itself at a geometric progression speed thus giving a \(CRN\).}

Also, we construct another sequence of rational numbers \(a_n^2\). We let this sequence to be a constant sequence, such that for each natural \(n\) we let \(a_n^2 = 3\).

\begin{thm}\label{theorem3}
    Let \(AM = \frac{a_n^1 + a_n^2}{2}\) and \(GM = \sqrt{a_n^1 a_n^2}\) Then, it is undecidable whether \(AM = GM\) or \(AM > GM\).
\end{thm}

\begin{proof}
    Recall that \(AM=GM\) if and only if all the inputs are equal. i.e. \(\frac{x_1+x_2+\cdots +x_n}{n}=\sqrt[n]{x_1x_2\cdots x_n}\) if and only if \(x_1=x_2=\cdots=x_n\) \cite{Bullen}.

    If \(P\) never terminates, then, from the above construction, \(a_n^1\) will converge to 3. Then, \(AM = GM\). If \(P\) does eventually terminate, then \(a_n^1\) will not converge to 3, but to a number slightly larger that 3. Then, \(AM > GM\).

    For the sake of contradiction, assume that there exists an algorithm \(Q(CRN_1, CRN_2)\) such that it takes in a pair of \(CRNs\) and can always decide if the \(AM\) of the pair of \(CRNs\) is equal to the \(GM\) of the pair of \(CRNs\) (i.e. \(AM=GM\)).

    Since \(AM=GM\) (of \(a_n^1\) and \(a_n^2\)) if only if \(P\) never terminates, algorithm \(Q\) can be used to decide whether \(P\) terminates or not and hence it is a decision program for \(P\) that we know cannot exist.
    
    From this result, it shows that whether \(AM=GM\) is undecidable.
\end{proof}

The above theorem shows that the \(AM-GM\) inequality is undecidable, and then now we will show that the general power mean inequality is also undecidable.

\begin{thm} \label{theorem4}
    Let \(p\) and \(q\) be rational numbers such that \(p<q\), then it is undecidable whether \(M_p(a_n^1, a_n^2)=M_q(a_n^1, a_n^2)\) or \(M_p(a_n^1, a_n^2)<M_q(a_n^1, a_n^2)\).
\end{thm}

\begin{proof}
    This proof follows a similar argument as the proof for Theorem 3.

    As mentioned earlier, \(M_p(x_1, x_2) = M_q(x_1, x_2)\) only if \(x_1 = x_2\) and \(M_p(x_1, x_2) < M_q(x_1, x_2)\) otherwise \cite{Hardy}. 

    Now, if the program \(P\) never terminates, then \(a_n^1\) will converge to \(3\), thus \(M_p(a_n^1, a_n^2)=M_q(a_n^1, a_n^2)\). But if the program \(P\) does eventually terminate, then \(a_n^1\) will converge to a value slightly larger than \(3\), so \(M_p(a_n^1, a_n^2)<M_q(a_n^1, a_n^2)\).

    For the sake of contradiction, assume that there exists an algorithm \(Q(CRN_1, CRN_2)\) such that it takes in a pair \(CRNs\) and can always decide if the \(M_p\) of the pair of \(CRNs\) is equal to the \(M_q\) of the pair of \(CRNs\) (i.e. \(M_p(CRN_1, CRN_2) = M_q(CRN_1, CRN_2)\)).

    Since \(M_p(a_n^1, a_n^2) = M_q(a_n^1, a_n^2)\) if and only if \(P\) never terminates, algorithm \(Q\) can be used to decide whether \(P\) terminates or not. However, this cannot be true, as the domain of \(P\) is undecidable. Thus, the assumption that program \(Q\) exists is false.

    From this result, it shows that whether \(M_p(a_n^1, a_n^2) = M_q(a_n^1, a_n^2)\) is undecidable.
\end{proof}

Now, we extend this argument into 2 non-constant \(CRNs\), although the programs used are not completely independent (one works based on the output of the other).

Let \(P_1\) be a (partially defined) algorithm that transforms natural numbers into 0 and 1 and that is not extendable to a totally defined algorithm \cite{Kleene, Vereshchagin}. Then let \(P_2\) be an algorithm that runs \(P_1\) one step late on the same input, meaning that if \(P_1\) terminates on step \(i\), then \(P_2\) will terminate on step \(i+1\).  If \(P_1\) does not terminates, then \(P_2\) will also not terminate. This implies that \(P_2\) has the same set of inputs that it will terminate on as \(P_1\). Note that the Halting Problem for \(P_1\) is undecidable and the set of numbers on which \(P_1\) gives output 1 (or output 0) is undecidable. Since \(P_2\) is essentially \(P_1\) but run 1 step late, its Halting problem is also undecidable, along with the set of numbers on which it outputs 1 (or 0).

\begin{construction}\label{construction5}
    For each natural \(n\) we define the following computer-generated sequences of rational numbers \(b_n^1\) and \(b_n^2\).
    \begin{itemize}
        \item For each $k$ we put $b_n^1(k)=5+2^{-k}$ if the algorithm $P_1$ did not yet finish working on input $n$ by step $k$.

        \item we put $b_n^1(k)=5+2^{-m}$ if $P_1$ has finished working on $n$ by step $k$. $m$ is the step number when $P_1$ finished working.
    \end{itemize}
    
    We define \(b_n^2\) identically, but with algorithm \(P_2\).
\end{construction}

{\em For each $n$, the sequences $b_n^1$ and $b_n^2$ both converge in themselves at a geometric progression speed thus giving \(CRNs\).}

Now, use these \(CRNs\) to prove the following theorem:

\begin{thm}\label{theorem6}
    Let \(p\) and \(q\) be rational numbers such that \(p<q\), then it is undecidable whether \(M_p(b_n^1, b_n^2)=M_q(b_n^1, b_n^2)\) or \(M_p(b_n^1, b_n^2)<M_q(b_n^1, b_n^2)\).
\end{thm}

\begin{proof}
    The proof is almost identical to the proof for Theorem 4.

    The difference is that the condition of \(b_n^1 = b_n^2\) is when \(P_1\) does not terminate (which implies that \(P_2\) does not terminate). This is because if \(P_1\) terminates on step \(k\), then the members from \(b_n^1(k)=5+2^{-k}\) and onwards will stay at this constant. Additionally, since \(P_1\) terminates on step \(k\), \(P_2\) will terminate on step \(k+1\), thus the members from \(b_n^2(k+1)=5+2^{-(k+1)}\) and onwards will stay at that constant. Since \(5+2^{-k} \ne 5+2^{-(k+1)}\), \(b_n^1 \ne b_n^2\).

    Using contradiction with the same program \(Q(CRN_1, CRN_2)\) from the proof of Theorem 4, it can be shown that this program cannot exist, and this proves that it cannot be decided whether \(M_p(b_n^1, b_n^2)=M_q(b_n^1, b_n^2)\) or \(M_p(b_n^1, b_n^2)<M_q(b_n^1, b_n^2)\).
\end{proof}

With this result, we can see that the general power mean inequality of these two \(CRNs\) is undecidable. Again, we can extend this argument even further into an arbitrary \(n\) algorithms and \(CRNs\).

Fo an integer \(m \ge 3\), let \(P_m\) be an algorithm that runs \(P_{m-1}\) one step late on the same input, meaning that if \(P_{m-1}\) terminates on step \(i\), then \(P_m\) will terminate on step \(i+1\).  If \(P_{m-1}\) does not terminates, then \(P_m\) will also not terminate. This implies that \(P_m\) has the same set of inputs that it will terminate on as \(P_{m-1}\) (and hence as \(P_1\)). Since \(P_m\) is essentially \(P_{m-1}\) but run 1 step late, its Halting problem is also undecidable, along with the set of numbers on which it outputs 1 (or 0).

\begin{construction}\label{construction7}
    For each natural \(n\) we define the following computer-generated sequences of rational numbers \(b_n^3\), \dots, \(b_n^m\).
    For each \(3\le i\le m\):
    \begin{itemize}
        \item For each $k$ we put $b_n^i(k)=5+2^{-k}$ if the algorithm $P_i$ did not yet finish working on input $n$ by step $k$.

        \item we put $b_n^i(k)=5+2^{-l}$ if $P_i$ has finished working on $n$ by step $k$. $l$ is the step number when $P_i$ finished working.
    \end{itemize}
\end{construction}

{\em For each i, \(b_n^i\) converges in itself at a geometric progression speed thus giving \(CRNs\).}

Now, we use these \(CRNs\) from Construction 5 and 7 to prove the following theorem:

\begin{thm}\label{theorem8}
    Let \(p\) and \(q\) be rational numbers such that \(p<q\), then it is undecidable whether \(M_p(b_n^1, b_n^2, \dots, b_n^m)=M_q(b_n^1, b_n^2, \dots, b_n^m)\) or \(M_p(b_n^1, b_n^2, \dots, b_n^m)<M_q(b_n^1, b_n^2, \dots, b_n^m)\).
\end{thm}

\begin{proof}
    As stated earlier, the generalized power mean equality, 
    
    \(M_p(x_1, x_2, \dots, x_n)=M_q(x_1, x_2, \dots, x_n)\), holds exactly when \(x_1=x_2=\cdots=x_n\) \cite{Hardy}.

    Now, for the \(CRNs\) \(b_n^1, b_n^2, \dots, b_n^m\) to be all equal, it must be that the corresponding programs, \(P_1, P_2, \dots, P_m\) must all not halt. As, if any of the program halts, it must be that all the other program halts at different steps, thus giving different \(CRNs\). A concrete example: if \(P_4\) halted on step 6, then \(P_5\) will halt on step 7; then \(b_n^4\) will converge to \(5+2^{-6}\) and \(b_n^5\) will converge to \(5+2^{-7}\). Thus, it will not be possible to have all of the computer-generated \(CRNs\) to be equal to each other if any of the program halts. Therefore, \(P_1\) must never halt (and thus all of the other programs as well).

    For the sake of contradiction, assume that there exists an algorithm \(Q(CRN_1, CRN_2, \dots, CRN_m)\) such that it takes in a set of \(CRNs\) and can always decide if the \(M_p\) of the set of \(CRNs\) is equal to the \(M_q\) of the set of \(CRNs\) (i.e. \(M_p(CRN_1, CRN_2, \dots, CRN_m)=M_q(CRN_1, CRN_2, \dots, CRN_m)\)).

    Since \(M_p(b_n^1, b_n^2, \dots, b_n^m)=M_q(b_n^1, b_n^2, \dots, b_n^m)\) if only if \(P_1\) never terminates, algorithm \(Q\) can be used to decide whether \(P_1\) (and all the other mentioned programs) terminates or not and hence it is a decision program for \(P_1\) that we know cannot exist.
    
    From this result, it shows that whether 
    
    \(M_p(b_n^1, b_n^2, \dots, b_n^m)=M_q(b_n^1, b_n^2, \dots, b_n^m)\) is undecidable.
\end{proof}

\section{Conclusions}
We have presented the proof of algorithmical undecidability of various means inequality, especially the power means inequality. The basis for such undecidability is the undecidability of the halting of programs that generates \(CRNs\), which are the parameters of the means.
If the parameters of the means are rational numbers, then it is easily algorithmically checked whether rational numbers are equal to one another, and thus is decidable. However, we have proved that it is not always so in a general power means inequality.

\end{document}